\documentclass{amsart}
\usepackage[all]{xy}
\usepackage{amsfonts,amssymb,amsmath,amscd}
\usepackage{amsthm}

\newcounter{zlist}
\newenvironment{zlist}{\begin{list}{{\rm(\arabic{zlist})}}{
\usecounter{zlist}\leftmargin2.5em\labelwidth2em\labelsep0.5em
\topsep0.6ex\itemsep0.3ex plus0.2ex minus0.3ex
\parsep0.3ex plus0.2ex minus0.1ex}}{\end{list}}

\newcounter{blist}
\newenvironment{blist}{\begin{list}{{\rm(\alph{blist})}}{
\usecounter{blist}\leftmargin2.5em\labelwidth2em\labelsep0.5em
\topsep0.6ex \itemsep0.3ex plus0.2ex minus0.3ex
\parsep0.3ex plus0.2ex minus0.1ex}}{\end{list}}

\newcounter{rlist}
\newenvironment{rlist}{\begin{list}{{\rm(\roman{rlist})}}{
\usecounter{rlist}\leftmargin2.5em\labelwidth2em\labelsep0.5em
\topsep0.6ex\itemsep0.3ex plus0.2ex minus0.3ex
\parsep0.3ex plus0.2ex minus0.1ex}}{\end{list}}

\swapnumbers
\newtheorem{theorem}{Theorem}[section]
\newtheorem{lemma}[theorem]{Lemma}
\newtheorem{thm}[theorem]{}
\newtheorem{proposition}[theorem]{Proposition}
\newtheorem{definition}[theorem]{Definition}

\numberwithin{equation}{section}

\newcommand{\A}{{\mathbb{A}}}
\newcommand{\B}{{\mathbb{B}}}

\newcommand{\II}{{\mathbb{I}}}
\newcommand{\V}{{\mathbb{V}}}

\newcommand{\bC}{\mathbf{C}}
\newcommand{\bG}{\mathbf{G}}
\newcommand{\bH}{\mathbf{H}}
\newcommand{\bT}{\mathbf{T}}

\newcommand{\bV}{{\mathbb {V}}}
\newcommand{\oH}{{\overline{H}}}
\newcommand{\woH}{{\widehat{\overline{H}}}}
\newcommand{\oR}{{\overline{R}}}
\newcommand{\oF}{{\overline{F}}}

\newcommand{\wT}{{\widehat{T}}}
\newcommand{\wG}{{\widehat{G}}}
\newcommand{\ot}{\otimes}
\newcommand{\uH}{{\underline{H}}}
\newcommand{\lra}{\longrightarrow}
\newcommand{\ve}{\varepsilon}
\begin{document}

\title{Notes on bimonads and Hopf monads}
 \author{Bachuki Mesablishvili and Robert Wisbauer}

\begin{abstract}

For a generalisation of the classical theory of Hopf algebra over fields,
A. Brugui\`eres and A. Virelizier study opmonoidal monads
on monoidal categories (which they called {\em bimonads}).
In a recent joint paper with S. Lack the same authors define the notion
of a {\em pre-Hopf monad} by requiring only a special form of the fusion
operator to be invertible.
In previous papers it was observed by the present authors that
bimonads yield a special case
%Hopf monads may be considered as a special case
of an entwining of a pair of functors (on arbitrary
categories). The purpose of this note is to show that in this setting the
pre-Hopf monads are a special case of Galois entwinings. As a byproduct some
new properties are detected which make a (general) bimonad on a Cauchy complete category to a Hopf monad. In the final section applications to cartesian monoidal
categories are considered. 
\smallskip


\end{abstract}

\maketitle

\tableofcontents

\section*{Introduction}

The classical definitions of bialgebras and Hopf algebras over fields (or rings)
heavily depend on constructions based on the tensor product. This may have
been one of the reasons why first generalisations of this notions were formulated
for monoidal categories, or even autonomous monoidal categories when
the properties of finite dimensional Hopf algebras were in the focus.
This was also the starting point for the definitions of {\em Hopf monads}
by I. Moerdijk in \cite{Moer}.
% by A. Brugui\`eres and A. Virelizier in \cite{BV}.
 McCrudden \cite{McC} suggested to call these functors
{\em opmonoidal monads} and A. Brugui\`eres and A. Virelizier just called
them {\em bimonads} in \cite[Section 2.3]{BLV}.

To be more precise, such a bimonad on a monoidal category
$(\bV, \otimes, \II)$ is a monad $\bT=(T,m,e)$ on $\bV$ endowed with natural
transformations $\chi: T\ot \to T\ot T$ and a morphism $\theta: T(\II)\to \II$
subject to certain (compatibility) conditions. These allow to define
left and right {\em fusion operators} by
$$\begin{array}{l}
H^l_{V,W}: (T(V)\ot m_W) \chi_{V, T(W)} : T(V\ot T(W))\lra T(V)\ot T(W), \\[+1mm]
H^r_{V,W}: (m_V \ot T(W)) \chi_{T(V), W}: T(T(V)\ot W)\lra T(V)\ot T(W).
\end{array}
$$

As a general form of the {\em Fundamental Theorem}
for Hopf algebras it is described in \cite[Theorem 4.6]{BV}
under which conditions the opmonoidal monads induce an equivalence
between the base (autonomous monoidal) category and the category of related
bimodules.

It was observed in \cite{MW1} (see also \cite{BBW}) that the notions around
Hopf algebras can be formulated for any category $\A$ without referring to
tensor products. For a {\em bimonad} on $\A$ one requires simply a monad
and a comonad structure whose compatibility is essentially
expressed by {\em distributive laws} (e.g. \cite[Definition 4.1]{MW1}).

As pointed out in \cite[Section 2.2]{MW1}, the opmonoidal monads yield special
cases of the entwining of a monad with a comonad on any category:
Hereby the monad $\bT$ is entwined with the comonad $-\ot T(\II)$.
 In \cite[Theorem 5.11]{MW} the above mentioned \cite[Theorem 4.6]{BV} is
formulated in terms of entwining functors.

In \cite{BLV} an opmonoidal monad (bimonad) is called a {\em Hopf monad}
provided the left and right fusion operators are isomorphisms and it is called a
{\em left ({\em resp.} right) pre-Hopf monad} if, for any $V\in \bV$, the morphisms $H^l_{\II,V}$ (resp. $H^r_{V,\II}$) is invertible.

In this paper we show that the right pre-Hopf monads $\bT$ are just
those for which the related entwining is $\bG_{\emph{T}(\II)}$-Galois
in the sense of \cite[3.13]{MW1}. This leads to an improved version of
\cite[Theorem 6.11]{BLV} which describes when a pre-Hopf monad on $\bV$
induces an equivalence between $\bV$ and the category of left Hopf $\bT$-modules
(see Theorem \ref{th.1}).

In Section 1 we recall some basic notions and can use \cite[Lemma 2.19]{BLV}
to improve some of our own results on Galois entwinings (see Theorem \ref{T.1}).

This is applied in Section 2 to find new properties of a bimonad
in the sense of \cite{MW1} to make it a Hopf monad, provided the base category is Cauchy complete.

In Section 3 opmonoidal monads $T$ on $(\bV,\ot,\II)$ are investigated.
In this case $T(\II)$ is a comonoid in $\bV$ and we have an entwining between
$\bT$ and $-\ot T(\II)$.
As mentioned above, the main result in this section is Theorem \ref{T.1}
which tells us when pre-Hopf monads induce an equivalence between $\bV$ and
$\bV^{G_{\emph{T}(\II})}_\emph{T}$.
We also observe (in \ref{BV-mod})  that for any $\bV$-comonoid
$\bC=(C,\delta,\ve)$, $T(C)$ also allows for a $\bV$-comonoid structure
provided $\bC$ allows for a grouplike morphism $g:\II\to C$.
In this case we get functors from $\bV$ to $\bV^{G_{\emph{T}(\bC})}_\emph{T}$ and
the question arises under which conditions these induce an equivalence.
It is shown in Theorem \ref{th.2} that this is only the case if
$g:\II\to C$ is an (comonad) isomorphism.

In the final section we consider applications to cartesian monoidal categories and provide examples of pre-Hopf functors for which the related comparison functor is not an equivalence.

\section{Preliminaries}

For a monad $\bT=(T, m, e)$ on a category
$\A$, we write $\A_T$ for the Eilenberg-Moore category of $\bT$-modules and write $$\eta_{T}, \varepsilon_{T}: \phi_{T}
\dashv U_{T} : \A_{T} \to \A$$ for the corresponding
forgetful-free adjunction. Dually, if ${\bG} =(G, \delta,
\varepsilon)$ is a comonad on $\A$,  we denote by $\A^G$ the Eilenberg-Moore category of $\bG$-comodules and by
$$\eta^{G} , \varepsilon^{G} : U^{G} \dashv \phi^{G} : \A \to \A^{G}$$
the corresponding forgetful-cofree adjunction.

For convenience we recall some notions and results from \cite[Section 3]{MW}.

\begin{thm}\label{M.F}{\bf Module functors.} \em
Given a monad $\bT=(T,m,e)$ on $\A$ and any functor $L : \A \to \B$, we say that $L$ is a
({\em left}) {\em$\mathbf{T}$-module} if there exists a natural transformation
$\alpha_L : TL \to L$ such that the diagrams
$$
\xymatrix{
L \ar@{=}[dr] \ar[r]^-{eL}&TL \ar[d]^-{\alpha_L}\\
& L ,} \qquad
\xymatrix{
TTL \ar[r]^-{mL} \ar[d]_-{T\alpha_L }& TL \ar[d]^-{\alpha_L}\\
TL \ar[r]_-{\alpha_L}& L}
$$ commute.

It is shown in \cite[Proposition II.1.1]{D} that a left $\bT$-module structure on $R$ is equivalent to the existence of a functor
$\overline{R} : \B \to \A_T$ inducing a commutative diagram

$$\xymatrix{ \B \ar[r]^\oR \ar[dr]_R & \A_T \ar[d]^{U_T} \\
           & \A.}$$

It is also shown in \cite {D} that,
 For any $\bT$-module
$(R: \B \to\A,\alpha)$ admitting a left adjoint functor $F : \A \to \B$, the
composite
$$ t_{\oR}: \xymatrix{T \ar[r]^-{T \eta}& TRF \ar[r]^-{\alpha F} & RF},$$
where $\eta : 1 \to RF$ is the unit of the adjunction $F \dashv R$,
is a monad morphism from $\bT$ to the monad on $\A$
generated by the adjunction $F \dashv R$.
\end{thm}

\begin{thm}\label{D.1.1}{\bf Definition.} \emph{(}\cite[2.19]{BBW}\emph{)} \em
A left $\bT$-module $R: \B \to \A$ with a left
adjoint $F: \A \to \B$ is said to be $\bT$-\emph{Galois} if the
corresponding morphism $t_{\overline{R}}: T \to RF$ of monads on
$\A$ is an isomorphism.
\end{thm}

Expressing the dual of \cite[Theorem 4.4]{M} in the present situation gives:

\begin{proposition}\label{P.1.1}
The functor $\overline{R}$ is an equivalence of categories if and
only if the functor $R$ is $\bT$-\emph{Galois} and monadic.
\end{proposition}

\begin{thm}\label{Cm.F}{\bf Comodule functors.} \em
Given a comonad $\bG=(G,\delta,\varepsilon)$ on $\A$, a functor $K: \B \to
\A$ is a {\em left $\mathbf{G}$-comodule functor} if there exists a
natural transformation $\beta_K : K \to GK$ inducing commutativity of the diagrams
$$
\xymatrix{
K \ar@{=}[dr] \ar[r]^-{\beta_K}& GK \ar[d]^-{\varepsilon K}\\
& K ,} \qquad
\xymatrix{
K \ar[r]^-{\beta_K} \ar[d]_-{\beta_K}& GK \ar[d]^-{\delta K}\\
GK \ar[r]_-{G \beta_K}& GGK.}
$$

A left $\bG$-comodule structure on $F : \B \to \A$ is equivalent to
the existence of a functor
$\overline{F} : \B \to \A^G$ (dual to \cite[Proposition II.1.1]{D})
leading to a commutative diagram
$$\xymatrix{ \B \ar[r]^\oF \ar[dr]_F & \A^G \ar[d]^{U^G}\\
  & \A .}$$

  If a $\bG$-comodule $(F,\beta)$
 admits a right adjoint $R: \A \to \B$, with
counit $\sigma : FR \to 1$, then there is a comonad morphism
$$\xymatrix{t_{\overline{F}}:FR \ar[r]^-{\beta R}& GFR
\ar[r]^-{G \sigma }  & G}$$
from the comonad generated by the adjunction $F \dashv R$ to the comonad $\bG$.
\end{thm}

\begin{thm}{\bf Definition.}\label{def} \em (\cite[Definition 3.5]{MW1})
 A left $\bG$-comodule $F: \B \to \A$ with a right
adjoint $R: \A \to \B$ is said to be $\bG$-\emph{Galois} if the
corresponding morphism $t_{\overline{F}}: FR \to G$ of comonads on
$\A$ is an isomorphism.
\end{thm}

  Now \cite[Theorem 2.7]{GoTo} (also \cite[Theorem 4.4]{M}) can be rephrased as follows:

\begin{proposition}\label{P.1.2} The functor $\overline{F}$ is an equivalence of categories
if and only if the functor $F$ is $\bG$-Galois and comonadic.
\end{proposition}

Recall \cite[Definition 1.19]{MW}:

\begin{thm}{\bf Definitions.}\label{def-galois} \em
Let $\bT=(T, m,e)$ be a monad and $\bG =(G,\delta,\varepsilon)$
a comonad on $\A$.
We say that $\bG$ is {\em $\bT$-Galois}, if there exists a left
$\bT$-module structure $\alpha: TG \to G$ on the functor $G$
such that the composite
$$\xymatrix{\gamma^G : TG \ar[r]^-{T\delta}& TGG \ar[r]^-{\alpha G}& GG}$$
is an isomorphism.

Dually, $\bT$ is {\em $\bG$-Galois}, if there is a left
$\bG$-comodule structure $\beta: T \to GT$ on the functor
$T$ such that the composite 
$$\xymatrix{\gamma_T : TT \ar[r]^-{\beta T }& GTT \ar[r]^-{G m}& GT}$$ 
is an isomorphism.
\end{thm}

\begin{thm}\label{Ent}{\bf Entwinings.} \em
Recall (for example, from \cite{Wf}) that an \emph{entwining} or
\emph{mixed distributive law} from a monad $\bT=(T,m,e)$ to a comonad
$\bG=(G, \delta,\varepsilon)$ on a category $\A$ is a natural
transformation $\lambda : TG \to TG$ with commutative diagrams
$$
\xymatrix{ & G \ar[dl]_{\eta G} \ar[rd]^{G \eta}& & & T G
\ar[dl]_{\lambda} \ar[rd]^{T \varepsilon}&\\
T G \ar[rr]_{\lambda}&& G T \,,& G T \ar[rr]_{\varepsilon T} &&
T,}$$
$$\xymatrix{ T G \ar[d]_{\lambda} \ar[r]^{T\delta} & T GG
\ar[r]^{\lambda G} & G T G \ar[d]^{G \lambda} \\
G T \ar[rr]_{\delta T}&& G G T ,}
\quad
\xymatrix{
 TT G\ar[d]_{\mu G} \ar[r]^{T \lambda} & T G T \ar[r]^{\lambda T} & G TT
 \ar[d]^{G\mu}\\
 TG \ar[rr]_{\lambda} && G T.}$$

It is well-known (see \cite{Wf}) that the following structures are in bijective
correspondence:
\begin{itemize}
\item entwinings $\lambda : TG \to GT$;
\item comonads $\wG=(\wG, \widehat{\delta},\widehat{\varepsilon})$ on $\A_{\bT}$ that extend $\bG$ in the
sense that
\begin{center}
$U_T \wG= G U_T$, $U_T \widehat{\delta}=\delta
U_T$ and $U_T
\widehat{\varepsilon}=\varepsilon U_T$;
\end{center}
\item monads $\wT=(\wT,\widehat{m},\widehat{e})$ on $\A^{\bG}$ that extend $\bT$ in the sense that
\begin{center}
$U^G \wT=T U^G$, $U^G \widehat{m}=m U^G$ and $U^G
\widehat{e}=e U^G$.
\end{center}
\end{itemize}

For any entwining $\lambda : TG \to GT$, $(a, h_a) \in \A_\bT$
and  $(a, \theta_a) \in\A^{\bG}$ (e.g. \cite[Section 5]{W}),
$$\wG(a,h_a)=(G(a), G(h_a) \cdot \lambda_a),\quad \widehat{\delta}_{(a, h_a)}=\delta_a,\quad \widehat{\varepsilon}_{(a,h_a)}=\varepsilon_a ,$$
$$\wT(a, \theta_a)=(T(a),
\lambda_a \cdot T(\theta_a)), \quad \widehat{m}_{(a, \theta_a)}=m_a,\quad
\widehat{e}_{(a, \theta_a)}=e_a. $$

We write $\A^{G}_{T}(\lambda)$ (or just $\A^{G}_{T},$ when $\lambda$ is understood) for the category whose
objects are triples $(a, h_a, \theta_a)$, where $(a, h_a) \in
\A_{T}$ and $(a,\theta_a) \in \A^{G}$ with commuting diagram
$$
\xymatrix{
T(a) \ar[r]^-{h_a} \ar[d]_-{T(\theta_a)}& a \ar[r]^-{\theta_a}& G(a) \\
TG(a) \ar[rr]_-{\lambda_a}&& GT(a). \ar[u]_-{G(h_a)}}
$$

The assignments $(a, h_a, \theta_a) \to ((a, h_a), \theta_a))$ and
$((a, h_a), \theta_a))\to ((a,\theta_a), h_a) $ yield  isomorphisms of categories
$$\A^{G}_{T}(\lambda) \simeq (\A _{T})^{\widehat{G}} \simeq (\A^{G})_{\widehat{T}}.$$
\end{thm}

We fix now an entwining $\lambda: TG \to GT$ and let $K: \A \to (\A^G)_{\widehat{T}}$ be a functor inducing commutativity of the diagram
$$
\xymatrix{ \A  \ar[r]^-{K}\ar[rd]_{\phi^G} & (\A^G)_{\widehat{T}} \ar[d]^{U_{\widehat{T}}}\\
& \A^G .}$$
Writing $\alpha_K : \widehat{T}\phi^G \to\phi^G$
for the corresponding  $\widehat{\bT}$-module structure
on $\phi^G$ (see \ref{M.F}), the natural transformation
$$U^G (\alpha_K): TG=TU^G\phi^G=U^G\widehat{T}\phi^G \to U^G\phi^G=G$$ 
provides a left $\bT$-module structure on $G$ (see \cite[Section 2]{MW}).

Similarly, if $K: \A \to (\A_T)^{\widehat{G}}$ is a functor inducing a commutative
diagram
$$
\xymatrix{ \A  \ar[r]^-{K}\ar[rd]_{\phi_T} & (\A_T)^{\widehat{G}} \ar[d]^{U^{\widehat{G}}}\\
& A_T ,}$$ then the natural transformation

$$ U_T (\beta_{K}): T=U_T \phi_T \to GT=GU_T\phi_T=U_T\widehat{G}\phi_T  ,$$ where $\beta_{K} :  \phi_T \to \widehat{G}\phi_T$
is the corresponding  $\widehat{\bG}$-comodule structure
on $\phi_T$ (see \ref{Cm.F}),
induces a $\bG$-comodule structure on $T$ (see again \cite[Section 2]{MW}).

The following part of \cite[Lemma 2.19]{BLV} is of use for our investigation.

\begin{lemma}\label{BLV-2.19} Let $\tau: F U_T \to F' U_T $ be a natural transformation, where $F,F':\A \to \B$ are arbitrary
functors. If the natural transformation $$\tau \phi_T : FT=FU_T\phi_T \to F'U_T\phi_T=F'T $$ is an isomorphism, then
so is $\tau$.
\end{lemma}

\begin{proposition}\label{P.1.3}
Suppose $K: \A \to (\A_T)^{\widehat{G}}$ to be a functor
with $ U^{\widehat{G}}K=\phi_T$ and
denote by $\beta_K : \phi_T \to \widehat{G}\phi_T$
the corresponding  $\widehat{\bG}$-comodule
structure on $\phi_T$. Then $(\phi_T,\beta_K)$ is $\widehat{\bG}$-Galois if and only if
$(T, U_T(\beta_K))$ is $\bG$-Galois.
\end{proposition}
\begin{proof} By \ref{def}, $(\phi_T,\beta_K)$ is $\widehat{\bG}$-Galois if the comonad morphism $t_K:\phi_TU_T \to \widehat{G}$, which is the composite $$\phi_TU_T\xrightarrow{\beta_KU_T}\widehat{G}\phi_TU_T
\xrightarrow{\widehat{G}\varepsilon^T} \widehat{G},$$ is an isomorphism, while, by \ref{def-galois}, $(T, U_T(\beta_K))$ is $\bG$-Galois if the composite
$$\xymatrix{\gamma_T : TT
\ar[r]^-{U_T\beta_K T }& GTT \ar[r]^-{G m}& GT}$$ is an isomorphism. So, we have to show that $t_K$ is an isomorphism if and only if $\gamma_T$ is so.

Since $U_T \widehat{G}=GU_T$, the natural transformation
$$U_Tt_K:U_T\phi_TU_T\xrightarrow{U_T\beta_KU_T}U_T\widehat{G}\phi_TU_T
\xrightarrow{U_T\widehat{G}\varepsilon^T}U_T \widehat{G}$$  can be rewritten as
$$TU_T\xrightarrow{U_T\beta_KU_T}GU_T\phi_TU_T \xrightarrow{GU_T\varepsilon^T}GU_T .$$
Then $U_Tt_K\phi_T$ is the composite $$TT=TU_T \phi_T\xrightarrow{U_T\beta_KU_T \phi_T}GU_T\phi_TU_T \phi_T\xrightarrow{GU_T\varepsilon^T \phi_T}GU_T \phi_T,$$ and since $U_T\varepsilon^T \phi_T=m :TT= U_T \phi_T U_T \phi_T \to U_T \phi_T=T ,$  it follows that $U_Tt_K\phi_T$ is just $\gamma_T.$ Now, if $t_K$ is an isomorphism, it is then clear that $\gamma_T=U_T(t_K)\phi_T$ is also an isomorphism. Conversely, if $\gamma_T$ is an isomorphism, then by Lemma \ref{BLV-2.19}, $U_Tt_K$ is also an isomorphism.
But since $U_T$ is conservative, $t_K$ is an isomorphism too. This completes the proof.
\end{proof}

Dually, one has

\begin{proposition}\label{P.1.4}
Suppose that $K : \A \to (\A^G)_{\widehat{T}}$ is a
functor with $U_{\widehat{T}} K=\phi^G$ and let $\alpha_K :
\widehat{T}\phi^G \to \phi^G$ be the corresponding
$\widehat{\bT}$-module structure on $\phi^G$.
Then $(\phi^G, \alpha_K)$ is
$\widehat{\bT}$-Galois if and only if $(G, U^G(\alpha_K))$ is
$\bT$-Galois.
\end{proposition}

In view of Propositions \ref{P.1.3} and \ref{P.1.4}, we get from Propositions \ref{P.1.1} and \ref{P.1.2}:

\begin{theorem}\label{T.1} In the situation of Proposition \ref{P.1.3},
the functor $K: \A \to (\A_T)^{\widehat{G}}$ is an equivalence of categories if
and only if  $(T, U_T(\beta_K))$ is $\bG$-Galois and the monad $\bT$ is of effective
descent type (i.e. the functor $\phi_T : \A \to \A_T$ is comonadic.)

Dually, in the situation of Proposition \ref{P.1.4}, the functor
$K : \A \to (\A^G)_{\widehat{T}}$ is an equivalence if and only if $(G, U^G(\alpha_K))$ is
$\bT$-Galois and the comonad $\bG$ is of effective codescent type (i.e. the functor $\phi^G : \A \to \A^G$ is monadic).
\end{theorem}

\begin{thm}\label{G.E}{\bf Galois entwinings.} \em
Let $\bT=(T,m,e)$ be a monad and $\bG=(G,\delta,\ve)$ a comonad on a
category $\A$ with an entwining  $\lambda:TG\to GT$.
If $G$ has a grouplike morphism $g:1\to G$ (in the sense of \cite[3.14]{MW}), then $T$ has two left
$G$-comodule structures given by
$$\xymatrix{ \quad gT:T\to GT }
\; \mbox{ and } \;\quad \xymatrix{\tilde g:  T  \ar[r]^{Tg}& TG \ar[r]^\lambda & GT},$$
and it was shown in \cite{MW} that the equaliser $(T^g, i)$ of these natural transformations admits the
structure of a monad in such a way that $i: T^g \to T$ becomes a monad morphism. We write
$i^* : \A_{\emph{T}} \to \A_{T^g}$ for the functor that takes an arbitrary $\bT$-algebra $(a,h_a)\in \A_T$ to
the $\bT^g$-algebra $(a, h_a \cdot i_a) \in \A_{T^g}$. When the category $\A_T$ admits coequalisers   of reflexive pairs  (which is
certainly the case if   $\A$ has coequalisers of reflexive pairs
and  the functor $T$ preserves them), $i^*$ has a left adjoint
$i_* :\A_{T^g} \to \A_T$. In this case, according to the results of \cite{MW}, there is a comparison functor
$\overline{i}:\A_{T^g} \to (\A_T)^{\widehat{G}}$ yielding
commutativity of the diagram
\begin{equation}\label{D.1}
\xymatrix{\A \ar@/^2pc/@{->}[rrrr]^{K_{g, \bG} }\ar[rrd]_{\phi_T} \ar[rr]^{\phi_{T^g}}& &
\A_{T^g}\ar[d]^{i_*} \ar[rr]^{\overline{i}} & &
(\A_T)^{\widehat{G}}  \ar[lld]^{U^{\widehat{G}}}\\
 & &  \A_T & &}
\end{equation}
where $U^{\widehat{G}} : (\A_T)^{\widehat{G}} \to \A_T$ is the
evident forgetful functor and $K_{g,\bG}:\A \to (\A_T)^{\widehat{G}}$ is the functor
that takes $a \in \A$ to $((T(a),m_a),\widetilde{g_a})\in (\A_T)^{\widehat{G}}$ (see \cite[Section 3]{MW}).

Let us write $\widetilde{G}$ for the comonad
generated by the adjunction $i^* \dashv i_*$ and write
\begin{itemize}
  \item $S_{K_{g,\bG}}: U_T\phi_T \to \widehat{G}$ 
    for the comonad morphism corresponding to the outer diagram in (\ref{D.1}),
  \item $S_{\phi_{T^g}}: \overline{G} \to \widetilde{G}$ for the comonad morphism
    corresponding to the left triangle in (\ref{D.1}),
  \item and $S_{\overline{i}}:\widetilde{G} \to \widehat{G}$ for the comonad
    morphism corresponding to the right triangle in (\ref{D.1}) that exists
    according to \cite[Proposition 1.20]{MW}.
\end{itemize}
\end{thm}

\begin{definition}\cite{MW} \em In the circumstances above, we call 
$(\bT, \bG, \lambda, g)$ a \emph{Galois entwining}
if the comonad morphism $S_{\overline{i}}:\widetilde{G} \to \widehat{G}$
is an isomorphism, or, equivalently, the functor $i_*$ is $\widehat{G}$-Galois. In this case
$g: 1 \to G$ is said to be a \emph{Galois group-like morphism}.
\end{definition}

\begin{theorem}\cite{MW}\label{Gal-entw} Let $\lambda : TG \to GT$ be an
entwining from a monad $\bT$ to a comonad $\bG$ on a
category $\A$. Suppose that $g: 1 \to G$ is a grouplike morphism
such that the corresponding functor $i^*:\A_{T} \to \A_{T^g}$
admits a left adjoint functor $i_*:\A_{T^g} \to \A_T$.
Then the comparison functor $\overline{i}:\A_{T^g}\to (\A_T)^{\widehat{G}}$
is an equivalence of categories if and only if $(\bT,
\bG, \lambda, g)$ is a Galois entwining and the functor
$i_*$ is comonadic.
\end{theorem}

\section{Bimonads}

The preceding results allow to formulate new conditions which
turn bimonads into Hopf monads.
Recall from \cite[Definition 4.1]{MW1} that
a  bimonad $\bH$ on any category $\A$ is an endofunctor $H : \A \to \A$
with a monad structure
 $\uH=(H, m, e)$, a comonad structure $\oH=(H, \delta, \varepsilon)$,
and an entwining    $\lambda:HH\to HH$ from the monad $\uH$ to the comonad $\oH$
 inducing commutativity of the diagrams
$$
 \xymatrix{
HH \ar@{->}@<0.5ex>[r]^-{\varepsilon H} \ar@ {->}@<-0.5ex>[r]_-{H\varepsilon}
\ar[d]_m & H \ar[d]^\varepsilon \\
    H \ar[r]^\varepsilon & 1 , } \qquad
\xymatrix{
1 \ar[r]^-{e} \ar[d]_-{e} & H \ar[d]^-{\delta}\\
H \ar@{->}@<0.5ex>[r]^-{eH} \ar@ {->}@<-0.5ex>[r]_-{He}& HH,}
 \qquad
\xymatrix{ 1\ar[r]^e \ar[dr]_= & H \ar[d]^\varepsilon\\
         & 1,}$$
$$\xymatrix{
HH \ar[r]^-{m} \ar[d]_-{H \delta}& H \ar[r]^-{\delta}& HH \\
HHH \ar[rr]_-{\lambda H}&& HHH  \ar[u]_-{Hm}.}
$$

Given a bimonad $\bH$, one has the comparison functor
$$K_H : \A \to \A_H^H=\A^{\overline{H}}_{\underline{H}}(\lambda),\quad
 a \mapsto (H(a), m_a, \delta_a) $$
with commutative diagrams
$$
\xymatrix{ \A \ar[r]^-{K_H} \ar[rd]_-{\phi_{\underline{H}}} &
 \A_H^H
\ar[d]^-{U^{\widehat{\overline{H}}}}  \simeq (\A_{\underline{H}})^{\widehat{\overline{H}}} \\
& \A_{\underline{H}},}  \qquad
 \xymatrix{
\A \ar[r]^-{K_H} \ar[rd]_-{\phi^{\overline{H}}} &
\A_H^H
\ar[d]^-{U_{\widehat{\underline{H}}}}\simeq (\A^{\overline{H}})_{\widehat{\underline{H}}}\\
& \A^{\overline{H}}.}$$

Writing $K_{\uH}$ (resp. $K_{\oH}$) for the composite $\A \xrightarrow{K_H} \A_H^H \simeq (\A_{\underline{H}})^{\widehat{\overline{H}}}$
(resp. $\A \xrightarrow{K_H} \A_H^H \simeq (\A^{\overline{H}})_{\widehat{\underline{H}}}$) and writing $\alpha_{K_{\uH}}$ (resp. $\alpha_{K_{\oH}}$) for the $\widehat{\overline{H}}$-comodule (resp. $\widehat{\underline{H}}$-module) structure on $\phi_{\uH}$ (resp. $\phi^{\oH}$) that exists by \ref{Cm.F} (resp. \ref{M.F}), we know from \cite[4.3]{MW1} that $U_{\uH}(\alpha_{K_{\uH}})=\delta: H \to HH$ and that $U^{\oH}(\alpha_{K_{\oH}})=m:HH \to H.$ It then follows from \ref{def-galois} that $\gamma_{\uH}:\uH \,\uH \to \oH\uH$ is the composite $$HH \xrightarrow{\delta H}HHH \xrightarrow{H m}HH,$$ while $\gamma^{\oH}: \uH \oH \to \oH \,\oH$ is the composite $$HH \xrightarrow{H \delta}HHH \xrightarrow{mH} HH.$$

Employing the notions considered above we have the following list of

\begin{thm}\label{T.2}{\bf Characterisations of Hopf monads.} For a bimonad $\bH$ on a Cauchy complete category $\A$, the following are equivalent:
\begin{blist}
\item  $(\phi_{\uH},\alpha_{K_{\uH}})$ is $\woH$-Galois, i.e., the composite $t_{K_{\uH}}:\phi_{\uH}U_{\uH}\xrightarrow{\alpha_{K_{\uH}}U_{\uH}} \widehat{\overline{H}}\phi_{\uH}U_{\uH} \xrightarrow{\widehat{\overline{H}}\varepsilon_{\uH}}\widehat{\overline{H}}$ is an isomorphism;

\item $(\phi^{\oH}, \alpha_{K_{\oH}})$ is $\widehat{\uH}$-Galois, i.e., the composite $t_{K_{\oH}}:\widehat{\uH}\xrightarrow{\widehat{\uH} \eta^{\oH}} \widehat{\uH} \phi^{\oH}U^{\oH}\xrightarrow{\alpha_{K_{\oH}}U^{\oH}} \phi^{\oH}U^{\oH}$ is an isomorphism;

\item  the unit $e:1\to H$ is a Galois grouplike morphism;

\item  the functor $K_H : \A \to \A^H_H$ {\em(}hence also $K_{\uH}:\A \to (\A_{\underline{H}})^{\widehat{\overline{H}}}$ and
$K_{\oH}: \A \to (\A^{\overline{H}})_{\widehat{\underline{H}}}${\em)} is an
equivalence of categories;

\item  $(H, m)$ is $\overline{H}$-Galois, i.e.,  $\gamma_{\uH}:HH \xrightarrow{\delta H}HHH \xrightarrow{H m}HH$ is an isomorphism;

\item  $(H, \delta)$ is $\underline{H}$-Galois, i.e.,   $\gamma^{\oH}:HH \xrightarrow{H \delta}HHH \xrightarrow{mH} HH$ is an isomorphism;

\item  $\bH$ has an antipode, i.e., there exists a natural transformation $S: H \to H$ with
$$m\cdot HS \cdot \delta=e \cdot \varepsilon=m\cdot SH \cdot \delta.$$
\end{blist}
\end{thm}

\begin{proof} (a), (c) and (d) are equivalent by \cite[4.2]{MW}, while (e), (f) and (g) are equivalent by \cite[5.5]{MW1}. Moreover, (a)$\Leftrightarrow$(e) follows by Proposition \ref{P.1.3} and (b)$\Leftrightarrow$(f) by Proposition \ref{P.1.4}.
\end{proof}

 \begin{thm}\label{ex.1}{\bf Example.} \em Let $(\V,\tau)$ be a \emph{lax braided} monoidal
category (see, for example, \cite{BLV}) and $\textbf{A}=(A, m,e, \delta, \varepsilon)$ a bialgebra in $\V$.
We write $H$ for the endofunctor $A \ot -:\V \to\V$. It is easy to verify directly, using the axioms of lax braidings, that the natural transformation
$\overline{\tau}=\tau_{A}\ot -: HH \to HH$ is a \emph{local prebraiding} (in the sense of \cite{MW1}) and that
$$(H, \overline{m}, \overline{e}, \overline{\delta}, \overline{\varepsilon}),$$ where $\overline{m}=m\ot -,$ $\overline{e}=e\ot -,$
$\overline{\delta}=\delta \ot -$ and $\overline{\varepsilon}=\varepsilon \ot -$, is a $\tau$-bimonad on $\V$. Then, according to
\cite[Section 6]{MW1}, the composite $\widetilde{\tau}=\overline{m}H \cdot H \overline{\tau}\cdot \overline{\delta}H$ is an entwining
from the monad $(H,\overline{m}, \overline{e})$ to the comonad $(H,\overline{\delta}, \overline{\varepsilon})$ that makes
$(H, \overline{m}, \overline{e}, \overline{\delta}, \overline{\varepsilon})$ a bimonad on $\V$. Writing $\V^\mathbf{A}_\mathbf{A}$ for the category
$\V^{\overline{H}}_{\underline{H}}(\widetilde{\tau})$, we get from Theorem \ref{T.2} the following generalisation of
\cite[Theorem 6.12]{MW1}:
\end{thm}
\begin{proposition} Let $(\V,\tau)$ be a lax braided category such that $\V$ is  Cauchy complete.
If $\mathbf{A}$ is a bialgebra in  $\V$, then the comparison functor
$$K: \V \to \V^\mathbf{A}_\mathbf{A},\quad V \mapsto
 (A\ot V, m \ot V, \delta \ot V), $$
 is an equivalence of
categories if and only if $\mathbf{A}$ is a Hopf algebra, that is, $\mathbf{A}$ 
has an antipode.
\end{proposition}

\section{Opmonoidal Monads}

\begin{thm}\label{op-mon}{\bf Pre-Hopf monads.} \em
Recall (for example, from \cite{McC}) that an \emph{opmonoidal functor} from a monoidal
category $(\bV, \otimes, \II)$ to a monoidal category $(\bV', \otimes', \II')$ is a triple
$(S, \chi, \theta)$, where $S: \bV \to \bV'$ is a functor, $\chi: S \otimes \to S \otimes' S$
is a natural transformation, and $\theta: S(\II) \to \II'$ is a morphism that are compatible
with the tensor structures. Note that opmonoidal functors $S$ take $\bV$-comonoids (i.e. comonoids in $\bV$) into
$\bV'$-comonoids in the sense that if $\bC=(C, \delta, \varepsilon)$ is a $\bV$-comonoid, then the triple
$S(\bC)=(S(C),\chi_{C,C} \cdot S(\delta) , \theta \cdot S(\varepsilon))$ is a $\bV'$-comonoid.

Recall also (again from \cite{McC}) that an \emph{opmonoidal monad}
on a monoidal category $(\bV, \otimes, \II)$ is a monad $\bT=(T,m, e)$ on the category $\bV$ whose
functor-part $T$ is an opmonoidal endofunctor together with
natural transformations
\begin{center}
$\chi_{V,W} : T(V\ot W)\to T(V)\ot T(W)$ for $V,W\in \V$
\end{center}
and a morphism $\theta: T(\II)\to \II$  that are compatible with the monad structure.

For example, it was pointed out in \cite{BV} that any bialgebra $\textbf{A}=(A,\mu,\eta,\delta,\varepsilon)$ in a braided
monoidal category $(\bV, \otimes, \II)$ with braiding $\tau_{V, W}: V \otimes W \to W \otimes V$ gives rise to an
opmonoidal $\bV$-monad $A \ot -$, where the natural transformation $\chi_{V,W} : A
\otimes V \otimes W \to A \otimes V \otimes A  \otimes W$ is the
composite
$$
\xymatrix{A \otimes V \otimes W \ar[rr]^-{\delta \otimes V \otimes
W}&& A \otimes A \otimes V \otimes W \ar[rr]^-{A \otimes \tau_{A,
V}\otimes W}&& A \otimes V \otimes A \otimes W \, ,}$$ while $\theta:A \to \II$ is just $\varepsilon$.

\bigskip

From now on we shall assume (actually without loss of generality by the coherence
theorem in \cite{Mc}) that all our monoidal categories are strict.

According to  \cite{BLV}, an opmonoidal monad $\bT=(T,m, e)$ on
the monoidal category $(\bV, \ot, \II)$ is {\em left pre-Hopf} \;
   if, for any object $V$ of $\bV$, the composite
$$H^l_{\II,V}:TT(V)=T(\II \ot T(V)) \xrightarrow{\chi_{\II, T(V)}}T(\II )\ot TT(V) \xrightarrow{T(\II) \ot m_V}T(\II) \ot T(V)$$
is an isomorphism, and $\bT$  is {\em right pre-Hopf} provided
 $$H^r_{V, \II}:TT(V)=T(T(V)\ot \II) \xrightarrow{\chi_{T(V),\II}}TT(V)\ot T(\II) \xrightarrow{m_V \ot T(\II) } T(V) \ot T(\II) $$ is an isomorphism.
$\bT$ is called a {\em pre-Hopf monad} if it is both left and right pre-Hopf.

For for any  $(V, h_V)\in \bV_T$ and $W \in \bV$, consider the morphisms
$$\mathbb{H}^r_{V,W}:
 T(V \ot W)\xrightarrow{\chi_{V,W}}T(V) \ot T(W) \xrightarrow{h_V \ot T(W)}V \ot T(W),$$
and for any $V\in \bV$ and $(W, h_W)\in \bV_T$, define
$$\mathbb{H}^l_{V,W}: T(V \ot W)\xrightarrow{\chi_{V,W}}T(V) \ot T(W) \xrightarrow{T(V)\ot h_W}T(V) \ot W.$$

It is shown in \cite{BLV} that, for any $V\in \bV$, $H^r_{-,V}$ (resp. $H^l_{V,-}$) is an isomorphism if and only if $\mathbb{H}^r_{-,V}$ (resp. $\mathbb{H}^l_{V,-}$) is so. In particular, $\bT$ is right (resp. left) pre-Hopf monad if and only if 
for any  $(V, h_V)\in \bV_T$, the morphism $\mathbb{H}^r_{V,\II}$ 
(resp. $\mathbb{H}^l_{\II,V}$) is an isomorphism.
\end{thm}

\begin{thm}\label{BV-mod}{\bf Entwined modules.} \em
Let $(\bV, \ot, \II)$ be a monoidal category and let $\bT=(T,m, e)$ be an opmonoidal monad on $\bV$.
As the functor $T$ is opmonoidal, for any $\bV$-comonoid $\bC=(C, \delta, \varepsilon)$, the triple
$T(\bC)=(T(C),\chi_{C,C} \cdot T(\delta) , \theta_\II \cdot T(\varepsilon))$ is also a $\bV$-comonoid.
In particular, the triple $T(\II)=(T(\II),\; \chi_{\II,\,\II},\;
\theta)$ is a $\V$-comonoid corresponding to the trivial $\bV$-comonoid $\textbf{I}=(\II, 1_{\II}, 1_{\II})$.
Given a $\bV$-comonoid $\bC$, we write $\bG_{\bC}$ for the comonad on $\bV$ whose functor part is
$G_{\bC}=- \otimes C$.

The compatibility axioms for $\bT$ ensure that
the natural transformation
$$\lambda^{\bC}_-:=H^l_{-,C}=(T(-) \otimes m_C)\cdot \chi_{-,\, T(C)}:
    T(-\ot T(C)) \to T(-)\ot T(C) %  TG_{T(\bC)} \to G_{T(\bC)}T
$$
is a mixed distributive law (entwining) from the monad $\bT$ to the comonad
$\bG_{T(\bC)}$ and the diagrams in \ref{Ent} come out as

 $$\xymatrix{ & V \!\!\ot T(C) \ar[dl]_{e_{V \ot T(C)}}  \ar[d]^-{e_V \ot T(C)} \\
T(V \!\!\ot T(C)\!) \ar[r]_-{\lambda^{\bC}_{V,T(\bC)}}  & T(V)\!\! \ot \!T(C),} \quad
\xymatrix{
**[l]T(V \!\!\ot \!T(C)) \ar[r]^-{T(V \ot T(\varepsilon))} \ar[d]_-{\lambda^{\bC}_{V,T(\bC)}}
     & T(V \!\!\ot \!T(\II))\ar[dr]^{T(V \ot \theta)} \\
  **[l]T(V) \!\!\ot \!T(C)\ar[r]_-{T(V) \ot T(\varepsilon)} &
   T(V) \!\!\ot \! T(\II)\ar[r]_{\quad T(V) \ot \theta} &  T(V) ,}$$

$$\xymatrix{
T(V\ot T(C))\ar[dd]_{\lambda^{\bC}_{V,T(\bC)}}\ar[rr]^-{T(V\ot T(\delta))}&&
  T(V\ot T(C\ot C))\ar[rr]^-{T(V \ot \chi_{C,C})}&& T(V\ot T(C) \ot T(C))
                   \ar[d]^{\lambda^{\bC}_{V\ot T(\bC), T(C)}} \\
 & &   &   &  T(V \ot T(C))\ot T(C) \ar[d]^{\lambda^{\bC}_{V,T(\bC)}\ot T(C)}\\
T(V)\ot T(C) \ar[rr]^{T(V)\ot T(\delta)} & &
 T(V) \ot T(C\ot C)\ar[rr]^-{T(V) \ot \chi_{C,C}} & & T(V) \ot T(C) \ot T(C),}
$$

$$\xymatrix{
**[l]T(T(V \ot T(C))) \ar[r]^-{T (\lambda^{\bC}_{V,T(\bC))}} \ar[d]_{m_{V\ot T(\bC)}} &
       T(T(V) \ot T(C)) % \ar[r]^{\lambda^{\bC}_{T(V),T(C)}}
        \ar[rr]^{\lambda^{\bC}_{T(V),T(\bC)}} & &
      TT(V) \ot T(C)\ar[d]^{m_V \ot T(C)}  \\
T(V \ot T(C)) \ar[rrr]^{\lambda^{\bC}_{V,T(\bC)}} & & & T(V)\ot T(C) .} $$

The {\em entwined $T(\bC)$-modules} are objects $V \in \V$ with a
$T$-module structure $h:T(V)\to V$ and a $T(\bC)$-comodule structure
$\rho:V\to V\ot T(C)$ inducing commutativity of the diagram
$$\xymatrix{
 T(V) \ar[rr]^h \ar[d]_{T(\rho)} & &
        V \ar[rr]^{\rho} & &  V\ot T(C) \\
T(V\ot T(C)) \ar[rr]_{\chi_{V,T(C)}}& & T(V)\otimes TT(C)
  \ar[rr]_{T(V)\otimes m_C}& &
                  T(V)\otimes T(C) \ar[u]_{h \otimes T(C)}.}$$
They form a category in an obvious way which we denote by $\V^{T(\bC)}_T$. It is clear that
$\V^{T(\bC)}_T$ is just the category $\bV_T^{G_{T(\bC)}}(\lambda_\bC)=(\bV_T)^{\widehat{G_{T(\bC)}}}.$

When $\bC=\textbf{I}$ is the trivial $\bV$-comonad, the entwined $T(\textbf{I})$-modules
are named {\em right Hopf \,$T$-modules} in \cite[Section 4.2]{BV}
(also \cite[6.5]{BLV}).

There is another description of the category $\V^{T(\bC)}_T$. Since $\textbf{T}$ is opmonoidal, $\V_T$ is a monoidal
category, and the
functor $\phi_T:\bV \to \bV_T$ is also opmonoidal. Then, for any $\V$-comonoid $\textbf{C}$, the triple
$$\phi_T(\bC)=((T(C),m_C),\,\chi_{C,C} \cdot T(\delta) ,\, \theta_\II \cdot T(\varepsilon)))$$ is a $\bV_T$-comonoid and
it is easy to see that the comonad $\widehat{G_{T(\bC)}}$ is just the comonad $\textbf{G}_{\phi_T(\bC)}$ and that the
category $\V^{T(\bC)}_T$ is just the category $(\V_T)^{\phi_T(\bC)}$. In particular, if
$\phi_T(\textbf{I})=((T(\II),m_{\II})\; \chi_{\II,\,\II},\;
\theta)$ is a $\V_T$-comonoid  corresponding to the trivial $\bV$-comonoid $\textbf{I}=(\II, 1_{\II}, 1_{\II})$,
then $\widehat{G_{T(\textbf{I})}}=\textbf{G}_{\phi_T(\textbf{I})}$ and $\V^{T(\textbf{I})}_T=(\V_T)^{\phi_T(\textbf{I})}$.

\begin{thm}\label{R}{\bf Remark.} \em
It follows from the results of \cite[5.13]{MW} that, for an arbitrary bialgebra $\textbf{A}=(A,\mu,\eta,\delta,\varepsilon)$
in a braided monoidal category $(\bV, \ot, \II)$, the following are equivalent:
\begin{enumerate}
  \item [(i)] the natural transformation
$$\lambda^{\textbf{I}}_{-} =H^l_{-,\II}:A \ot -\ot A \to A \ot -\ot A,$$ corresponding to the
opmonoidal $\bV$-monad $A \ot -$, is an isomorphism;
  \item [(ii)] the composite
$$\lambda^{\textbf{I}}_\II=H^l_{\II,\II}:A \ot A \to A \ot A$$ is an isomorphism;
  \item [(iii)] the composite
$$A\ot A \xrightarrow{\delta \ot A} A\ot A \ot A \xrightarrow{A \ot m} A \ot A$$
is an isomorphism.
\end{enumerate}

\noindent Recall (for example from \cite{M}) that Condition (iii) is in turn equivalent to saying that
$\textbf{A}$ has an antipode, i.e. $\textbf{A}$ is a Hopf algebra. It follows from the equivalence
(i)$\Leftrightarrow$ (iii) that, for any $V \in \bV$, the natural transformation $H^l_{-, V}$,
which is easily seen to be just the natural transformation $H^l_{-, \II}\ot V$, is an
isomorphism, or equivalently, the monad  $A \ot -$ is left Hopf, if and only if $\textbf{A}$ is a Hopf algebra.
Moreover, if the monad $A \ot -$ is left pre-Hopf (and hence, in particular, the morphism $H^l_{\II,\II}$
is an isomorphism), then according to the equivalence (ii)$\Leftrightarrow$ (iii),
$\textbf{A}$ is a Hopf algebra. Putting this information together and using that, quite obviously,
any left Hopf monad is left pre-Hopf, we have proved that
the following are equivalent:
\begin{enumerate}
  \item [(i)] $\textbf{A}$ is a Hopf algebra;
  \item [(ii)] the monad  $A \ot -$ is left pre-Hopf;
  \item [(iii)] the monad  $A \ot -$ is left Hopf.
\end{enumerate} this result may be compared with \cite[Proposition 5.4(a)]{BLV}.
\end{thm}

\begin{thm}\label{grouplike}{\bf Grouplike morphisms.} \em
Suppose now that the $\bV$-comonoid $\bC$ allows for a grouplike element
$g:\II \to C$ (see  \cite{M}, \cite{MW1}).
Then direct inspection shows that the composite 
$$\overline{g}:\II \xrightarrow{\;g\;} C \xrightarrow{e_C} T(C)$$ is a
grouplike element for the $\bV$-comonoid $T(\bC)$ implying that the natural transformation
$$-\ot \overline{g}:1 \to - \ot T(C)$$ is a grouplike morphism. Thus the results of \cite{M} apply.
In particular, the composite
$$T(-) \xrightarrow{T(- \ot \overline{g})}T(- \otimes T(C)) \xrightarrow{\lambda^{\bC}_-} T(-) \otimes T(C)$$
gives the structure $\vartheta: \phi_T \to \phi_T \widehat{G_{T(\bC)}}$ of a $\widehat{G_{T(\bC)}}$-comodule
on the functor $\phi_T :\bV \to \bV_T.$ Since in the diagram
$$
\xymatrix{
T(-) \ar[rr]^-{T(-\ot g)}\ar[d]_{\chi_{-, \II}}&&T(-\ot C) \ar[rr]^-{T(-\ot e_C)}\ar[d]^{\chi_{-, C}}
&& T(-\ot T(C))\ar[rr]^-{\chi_{-, T(C)}}&& T(-)\ot T^2(C)\ar[d]^{T(-)\ot m_C}\\
T(-)\ot T(\II)\ar[rr]_{T(-)\ot T(g)}&& T(-)\ot T(C)\ar[rrrru]|{T(-)\ot T(e_C)} \ar@{=}[rrrr]&&&& T(-) \ot T(C)}
$$
the rectangle and the top triangle commute by naturality of $\chi$, while the bottom triangle commutes since $e$
is the unit for the multiplication $m$, it follows that $\vartheta$ is just the natural transformation
$$T(-) \xrightarrow{\chi_{-,\II}}T(-) \otimes T(\II) \xrightarrow{T(-) \ot T(g)} T(-) \otimes T(C).$$
It then follows that the assignment $V \longrightarrow ((T(V), m_V), (T(V) \ot T(g))\cdot \chi_{V,\II})$
yields a functor $K_{g,\bC}:=K_{g,G_{T(\bC)}}:\bV \to \V^{G_{T(\bC)}}_T$ leading to the commutative diagram
$$\xymatrix{\V \ar[rr]^-{K_{g,\bC}} \ar[rrd]_{\phi_T}& &
\V^{T(\bC)}_T=(\bV_T)^{\widehat{G_{T(\bC)}}} \ar[d]^{ U^{\widehat{G_{T(\bC)}}} }\\
 & & \V_T\,.}$$
One then calculates that for any  $(V, h_V)\in V_T$, the $(V, h_V)$-component
of the induced comonad morphism $S_{K_{g,\bC}}:\phi_TU_T \to \widehat{G_{T(\bC)}}$
is the composite
$$\xymatrix{T(V) \ar[r]^-{\chi_{V,\,\II}}& T(V) \otimes T(\II)
\ar[rr]^-{T(V) \otimes T(g)}&& T(V) \otimes T(C) \ar[rr]^{h_V\ot T(C)} && V \otimes T(C).}$$
In particular, when $\bC$ is the trivial $\bV$-comonoid $\textbf{I}$ together with the evident grouplike morphism
$1_\II : \II \to \II$, the morphism $\chi_{-,\II}:T(-) \to T(-)\ot T(\II)$ gives the structure
$\vartheta': \phi_T \to \phi_T \widehat{G_{T(\textbf{I})}}$ of a $\widehat{G_{T(\textbf{I})}}$-comodule
on the functor $\phi_T :\bV \to \bV_T,$ and then one has the following commutative diagram
$$\xymatrix{\V \ar[rr]^-{K_{1_{\II},\textbf{I}}} \ar[rrd]_{\phi_T}& &
\V^{T(\textbf{I})}_T=(\bV_T)^{\widehat{G_{T(\mathcal{\textbf{I}})}}} \ar[d]^{ U^{\widehat{G_{T(\mathcal{\textbf{I}})}}} }\\
& & \V_T }$$
with the comparison functor $K_{1_{\II},\textbf{I}}(V)=((T(V), m_V), \chi_{V,\, I})$. Moreover, for any  $(V, h_V)
\in V_T$, the $(V, h_V)$-component of the induced comonad morphism $S_{K_{1_{\II},\textbf{I}}}:\phi_TU_T \to \widehat{G_{T(\textbf{I})}}$ is the
composite
$$\xymatrix{T(V) \ar[r]^-{\chi_{V,\,\II}}& T(V) \otimes T(\II)
\ar[rr]^-{h_V \otimes T(\II)}&& V \otimes T(\II).}$$
Comparing now $\vartheta$ and $\vartheta'$ gives:
\begin{equation}\label{E.1}
\vartheta=(T(-) \otimes T(g)) \cdot \vartheta'
\end{equation} while comparing $S_{K_{g,C}}$ and $S_{K_{e_{\II},\II}}$ and using that
$$(h_V\ot T(C))\cdot (T(V) \otimes T(g))=(V\ot T(g))\cdot (h_V \otimes T(\II))$$ by
bifunctoriality of the tensor product, gives:
\begin{equation}\label{E.2}
S_{K_{g,C}}=(- \otimes T(g)) \cdot S_{K_{e_{\II},\II}}.
\end{equation}
\end{thm}

It is easy to see that $S_{K_{e_{\II},\II}}$ just the composite $\mathbb{H}^r_{V,\II}$.
This yields in particular a fact proved in \cite[Lemma 6.5]{BLV}:

\begin{lemma} \label{p.3.1} The natural transformation $\mathbb{H}^r_{-,\II}:\emph{T}(-)\to -\otimes \emph{T}(\II)$
is a morphism of comonads $\phi_\emph{T}\emph{U}_\emph{T} \to \widehat{\emph{G}_{\emph{T}(\emph{\textbf{I}})}}.$
\end{lemma}

We already know (see \ref{op-mon}) that $\bT$ is a right pre-Hopf monad iff the natural transformation
$\mathbb{H}^r_{-,\II}$ (or, equivalently, the comonad morphism $S_{K_{e_{\II},\II}}$) is an isomorphism.
It now follows from Proposition \ref{P.1.3}:

\begin{proposition} \label{p.3.2} An opmonoidal monad $\bT$ on $\bV$ is a right pre-Hopf monad if and only if $\bT$ is
$\bG_{\emph{T}(\emph{\textbf{I}})}$-Galois.
\end{proposition}

This allows us to present an improved version of \cite[Theorem 6.11]{BLV}.

\begin{theorem}\label{th.1} For an opmonoidal monad $\bT$ on a monoidal category $(\bV, \otimes, \II)$, the following are equivalent:
\begin{blist}
  \item   the functor $K_{1_{\II},\emph{\textbf{I}}}: \bV \to
 \V^{\emph{T}(\emph{\textbf{I})}}_{\emph{T}}$ is an
 equivalence of categories;
  \item   \begin{rlist}
                  \item $\bT$ is $G_{\emph{T}(\emph{\textbf{I})}}$-Galois,
                  \item $\bT$ is of effective descent type.
                \end{rlist}
\end{blist}
\end{theorem}
\begin{proof} The assertion follows by Proposition \ref{P.1.2}.
\end{proof}

\begin{theorem}\label{th.2} Let $\bT=(T,m,e)$ be an opmonoidal monad on a monoidal category
$(\bV, \otimes, \II)$ and $\bC=(C, \delta,\varepsilon)$ a $\bV$-comonoid with a grouplike element $g: \II \to C$.
The following are equivalent:
\begin{blist}
  \item   the functor $K_{g,\bC}:\bV  \to
(\V_{\emph{T}})^{^{\phi_\emph{T}(\emph{\textbf{C})}}}$
is an equivalence of categories;
  \item   $K_{1_{\II},\emph{\textbf{I}}}: \bV \to
     (\V_{\emph{T}})^{^{\phi_\emph{T}(\emph{\textbf{I})}}}$
is an equivalence of categories
  and $g$ is an isomorphism;
  \item   \begin{rlist}
                  \item $\bT$ is $\emph{G}_{\emph{T}(\emph{\textbf{I})}}$-Galois,
                  \item $\bT$ is of effective descent type,
                  \item $g$ is an isomorphism.
                \end{rlist}
\end{blist}
\end{theorem}
\begin{proof} Note first that, being grouplike morphisms, $g$ and $T(g)$ are
 split monomorphisms. Hence
the natural transformation $-\otimes T(g): G_{T(\II)} \to G_{T(C)}$ is also a split monomorphism.

Now, if the functor $K_{g,C}:\bV \to \V^{\emph{T}(\emph{\textbf{C})}}_{\emph{T}}=(\V_T)^{\widehat{G_{T(\bC)}}}$ is an equivalence of categories, then it follows
from Proposition \ref{P.1.2} that the monad $\bT$ is of effective descent type and the comonad morphism
$S_{K_{g,\bC}}:\phi_TU_T \to \widehat{G_{T(C)}}$ is an isomorphism. Since
$S_{K_{g,C}}=(- \otimes T(g)) \cdot S_{K_{e_{\II},\II}}$ by  (\ref{E.2}) and since the natural transformation
$-\otimes T(g): G_{T(\textbf{I})} \to G_{T(\bC)}$ is a split monomorphism, it follows that the natural
transformations $- \otimes T(g)$ and $S_{K_{e_{\II},\II}}$ are both isomorphisms. Then, in particular,
$T(g)$ is an isomorphism. Since $\bT$ is of effective descent type, the functor $T$ is conservative (see, \cite[Proposition 3.11]{Me}).
Thus $g$ is also an isomorphism. Since $\bT$ is of effective descent type and since $S_{K_{e_{\II},\II}}$ is
an isomorphism, it follows from Proposition \ref{P.1.2} that the functor $K_{e_{\II},\II}:
\bV \to \V^{\emph{T}(\emph{\textbf{\textbf{I}})}}_{\emph{T}}$ is an equivalence of categories. Hence (a) implies (b). Since (b) trivially
implies (a), (a) and (b) are equivalent. Finally, (b) and (c) are equivalent by Theorem \ref{th.1}.
\end{proof}
\end{thm}

\begin{thm}\label{submonad}{\bf Galois group-like morphisms.} \em
We will assume from now on that our monoidal category $(\bV, \ot,\II)$ admits equalisers.

Let $\bT=(T,m,e)$ be an opmonoidal monad on $\bV$, $\textbf{C}$ a $\bV$-comonoid and $g:\II \to \bC$ a
grouplike element. Since $\bV$ has equalisers, one can consider the $\bV$-monad $\bT^{-\ot \overline{g}}$.
We write $\bT^g$ for this monad. Let us say that $g:\II \to C$ is a Galois grouplike element if the induced
grouplike morphism $- \ot \overline{g}:1 \to G_{T(\textbf{C})}$ is Galois. In particular, the grouplike element
$1_\II :\II \to \II$ is Galois if the grouplike morphism $-\ot \overline{1_{\II}}=- \ot e_C:1 \to G_{T(\textbf{I})}$
is Galois.

\bigskip

Specialising now Theorem \ref{Gal-entw} to the present situation gives:

\begin{theorem}\label{Gal-Entw}Let $\bT$ be an opmonoidal monad on a monoidal category $(\bV, \ot, \II)$
and $\textbf{C}$ a $\bV$-comonoid. Suppose that $g: \II \to C$ is a grouplike element
such that the corresponding functor $i^*:\bV_{T} \to \bV_{T^g}$
admits a left adjoint functor $i_*:\bV_{T^g} \to \bV_T$.
Then the comparison functor $\overline{i}:\bV_{T^g}\to \bV_{T^g}\to \bV_T^{T(\bC)}=(\bV_T)^{\widehat{G_{T(\bC)}}}$
is an equivalence of categories if and only if $g$ is a Galois grouplike element and the functor
$i_*$ is comonadic.
\end{theorem}

Direct inspection shows (see also \cite[Section 5]{MW}) that $T^{1_{\II}}$ is given by the equaliser
$$\xymatrix{ T^{1_{\II}}(-) \ar[r]&  T(-)
\ar@{->}@<0.5ex>[rr]^-{T(-) \ot e_{\II}} \ar@{->}@<-0.5ex>
[rr]_-{\chi_{-, \II}}&& T(-)\otimes T(\II),}$$ while $T^g$ is given by the equaliser
$$\xymatrix{ T^g(-) \ar[r]&  T(-)
\ar@{->}@<0.5ex>[rr]^-{T(-) \ot e_{\II}} \ar@{->}@<-0.5ex>
[rr]_-{\chi_{-, \II}}&& T(-)\otimes T(\II) \ar[rr]^-{T(-) \ot T(g)}&& T(-) \ot T(C).}$$  Since $g$, being a
grouplike morphism, is a split monomorphism, so too is the natural transformation $T(-) \ot T(g)$. It follows that
the monad $\textbf{T}^{1_{\II}}$ can be identified with the monad $\textbf{T}^g$. Since $g: \II \to C$ is nothing but a comonoid
morphism from the trivial $\bV$-comonoid $\textbf{I}$
to the $\bV$-comonoid $\bC$ and since any opmonoidal functor preserves comonoid morphisms,
$T(g): T(\II) \to T(C)$ can be seen as a morphism of $\bV$-comonoids
$T(\textbf{I}) \to T(\bC)$. It is then easy to see  that the induced morphism of $\bV$-comonads
$- \otimes T(g): G_{T(\II)} \to G_{T(C)}$ can be lifted to a morphism
$\widehat{- \ot T(g)}:\widehat{G_{T(\textbf{I})}}\to \widehat{G_{T(\bC)}}$ of $\bV_T$-comonads. Using
that $\vartheta=(T(-) \otimes T(g)) \cdot \vartheta'$ by (\ref{E.1}), it follows from \cite[Lemma 3.9]{MW}
that one has the following commutative diagram

$$
\xymatrix{i_* \ar[rr]^{\alpha}\ar[dr]_{\alpha'}&& \widehat{G_{T(\bC)}} \cdot i_*\\
&\widehat{G_{T(\textbf{I})}}\cdot i_*\ar[ru]_{\widehat{-\ot T(g)}\cdot i_*}} $$ where $i:T^g=T^{1_{\II}}\to T$
is the canonical inclusion, while $\alpha$ (resp. $\alpha'$) is a left
$\widehat{G_{T(\bC)}}$-comodule (resp. $\widehat{G_{T(\textbf{I})}}$-comodule) structure on $i_*$. It then follows that
one also has commutativity in

\begin{equation}\label{D.2}
\xymatrix{\widetilde{G_{T(\bC)}} \ar[rr]^{S_{\overline{i}}}\ar[dr]_{S_{\overline{i}}}&& \widehat{G_{T(\bC)}}\\
&\widehat{G_{T(\textbf{I})}}\ar[ru]_{\widehat{-\ot T(g)}}, } \end{equation}
where $\widetilde{G_{T(\bC)}}$ denotes the comonad generated by the adjunction $i^* \vdash i_*$ (see \ref{G.E}).
\end{thm}

\begin{proposition}\label{Entw.} Let $\bT$ be an opmonoidal monad on a monoidal category $(\bV, \ot, \II)$
and $\textbf{C}$ a $\bV$-comonoid.
\begin{zlist}
  \item   A grouplike element $g:\II \to C$ is Galois if and only if
the grouplike element $1_\II :\II \to \II$ is Galois and
the morphism $T(g): T(\II)\to T(C)$ is an isomorphism.
 \item  If the monad $\bT$ is conservative, then any $\bV$-comonoid admitting 
   a Galois grouplike element is (isomorphic to) the trivial $\bV$-comonoid
    $\emph{\textbf{I}}$.
\end{zlist}
\end{proposition}
\begin{proof} To say that the grouplike morphism $g:\II \to C$ (resp.
$1_\II: \II \to \II$) is Galois is to say that the comonad morphism
$S_{\overline{i}}:\widetilde{G_{T(\textbf{I})}}\to \widehat{G_{T(\textbf{I})}}$
(resp. $S_{\overline{i}}:\widetilde{G_{T(\bC)}}\to \widehat{G_{T(\textbf{C})}}$)
is an isomorphism. Now, since $T(g)$ is a split monomorphism, the result follows from the
commutativity of Diagram (\ref{D.2}). This proves (1).

Recalling that a monad is called conservative provided that its functor-part is conservative, one sees
that (2) follows from (1).
\end{proof}

Combining Theorem \ref{Gal-Entw} and Proposition \ref{Entw.} gives:

\begin{theorem}\label{Main} Let $\bT$ be an opmonoidal monad on a monoidal category $(\bV, \ot, \II)$
such that the functor $i^*:\bV_{T} \to \bV_{T^{1_\II}}$ admits a left adjoint functor 
$i_*:\bV_{T^{1_\II}} \to \bV_T$ and $\textbf{C}$ a $\bV$-comonoid. Then, for any grouplike element 
$g: \II \to C$, the following are equivalent:
\begin{blist}
  \item   $g:\II \to C$ is a Galois grouplike element and the functor
$i_*$ is comonadic;
  \item   the comparison functor 
   $\overline{i}:\bV_{T^g}\to  
 (\V_{\emph{T}})^{^{\phi_\emph{T}(\emph{\textbf{C})}}}$
%\bV_T^{T(\bC)}$
is an equivalence of categories;
\item $1_\II :\II \to \II$ is a Galois grouplike element, the functor
$i_*$ is comonadic and the morphism $T(g): T(\II)\to T(C)$ is an isomorphism;
  \item   the comparison functor $\overline{i}:\bV_{T^{1_{\II}}}\to 
(\V_{\emph{T}})^{^{\phi_\emph{T}(\emph{\textbf{I})}}}$
 %\bV_T^{T(\textbf{I})}$
is an equivalence of categories and the morphism $T(g): T(\II)\to T(C)$ is an isomorphism.
\end{blist}
\end{theorem}

It is easy to see that, in the case where the monad $\textbf{T}^{1_\II}=\textbf{T}^g$ is (isomorphic to) the identity
monad, the functor $\phi_{T^{1_\II}}=\phi_{T^g}$ is (isomorphic to) the identity functor, the functor $i_*$
is (isomorphic to) the functor $\phi_T$, while the functor $\overline{i}$ is (isomorphic to) the
comparison functor $K_{g,\bC}$. Using now that the monad $\bT$ is conservative provided that
the functor $\phi_T$ is so, in the light of Proposition \ref{Entw.}, we get from Theorems \ref{th.2} and \ref{Main}:

\begin{theorem}\label{Main1} Let $\bT$ be an opmonoidal monad on a monoidal category $(\bV, \ot, \II)$
such that the monad $\textbf{T}^{1_\II}$ is (isomorphic to) the identity
monad and $\textbf{C}$ a $\bV$-comonoid. Then, for any grouplike element
$g: \II \to C$, the following are equivalent:
\begin{blist}
  \item   $g:\II \to C$ is a Galois grouplike element and the functor
$\phi_T$ is comonadic;
  \item   the comparison functor $K_{g,\bC}:\bV  \to
(\V_{\emph{T}})^{^{\phi_\emph{T}(\emph{\textbf{C})}}}$
% \V^{\emph{T}(\emph{\textbf{C})}}_{\emph{T}}$
is an equivalence of categories;
  \item   $1_\II :\II \to \II$ is a Galois grouplike element, the functor
$\phi_T$ is comonadic and the morphism $g:\II \to C$ is an isomorphism;
  \item   the comparison functor 
 $K_{g,\textbf{I}}:\bV\to  
 (\V_{\emph{T}})^{^{\phi_\emph{T}(\emph{\textbf{I})}}}$
%\bV_T^{T(\textbf{I})}$
is an equivalence of categories and the morphism $g:\II \to C$ is an isomorphism;
   \item    \begin{rlist}
                  \item $\bT$ is $\emph{G}_{\emph{T}(\emph{\textbf{I})}}$-Galois,
                  \item $\bT$ is of effective descent type,
                  \item $g$ is an isomorphism.
                \end{rlist}
\end{blist}
\end{theorem}

\begin{thm}\label{Aug}{\bf Definition.} \em
We say that an opmonoidal monad $\bT$ on a monoidal category $\bV$ is \emph{augmented} if it is equipped 
with a monad morphism $\sigma: T \to 1_{\bV}.$ In this case $\sigma$ is said to be an \emph{augmentation}.
\end{thm}

\begin{lemma}\label{aug} Suppose that $\bT$ is an augmented right-Hopf opmonoidal monad on a monoidal category $\bV$ with
an augmentation $\sigma: T \to 1_{\bV}.$ Then, for any $V \in \bV$, the composite
$$\overline{\sigma}_V: T(V) \xrightarrow{\chi_{V, \II}}T(V) \ot T(\II) \xrightarrow{\sigma_V \ot T(\II)} V \ot T(\II)$$
is an isomorphism.
\end{lemma}
\begin{proof}Just note that, since $\sigma: T \to 1_{\bV}$ is a morphism of monads, for any $V \in \bV$, $(V,\sigma_V)$ is an object
of $\bV_T$.
\end{proof}

\begin{theorem}\label{Main2} Let $\bT=(T, m, e)$ be an augmented right-Hopf opmonoidal monad on a 
Cauchy complete monoidal category $(\bV, \ot, \II)$
with an augmentation $\sigma: T \to 1_{\bV}.$ Then, for any grouplike element
$g: \II \to C$, the following are equivalent:
\begin{blist}
 \item   $g:\II \to C$ is a Galois grouplike element;
  \item   the comparison functor 
   $K_{g,\bC}:\bV \to 
    (\V_{\emph{T}})^{^{\phi_\emph{T}(\emph{\textbf{C})}}}$
%\V^{\emph{T}(\emph{\textbf{C})}}_{\emph{T}}$
is an equivalence of categories;
  \item   the comparison functor 
   $K_{g,\textbf{I}}:\bV\to   
   (\V_{\emph{T}})^{^{\phi_\emph{T}(\emph{\textbf{I})}}}$
  is an equivalence of categories and the morphism $g:\II \to C$ is an isomorphism;
\item $1_\II :\II \to \II$ is a Galois grouplike element and the morphism 
      $g:\II \to C$ is an isomorphism;
\item   \begin{rlist}
         \item $\bT$ is $\emph{G}_{\emph{T}(\emph{\textbf{I})}}$-Galois,
         \item $g$ is an isomorphism.
        \end{rlist}
\end{blist}
\end{theorem}
\begin{proof}Using naturality of $\chi$, it is not hard to check that the diagram
$$\xymatrix{T(V) \ar[d]_{\overline{\sigma}_V}
\ar@{->}@<0.5ex>[rr]^-{T(V) \ot e_{\II}} \ar@{->}@<-0.5ex>
[rr]_-{\chi_{V, \II}}&& T(V)\otimes T(\II) \ar[d]^{\overline{\sigma}_V \ot T(\II)}\\
V \ot T(\II)
\ar@{->}@<0.5ex>[rr]^-{V \ot T(\II) \ot e_{\II}} \ar@{->}@<-0.5ex>
[rr]_-{V \ot \chi_{\II, \II}}&& V \ot T(\II)\otimes T(\II)}$$
commutes. Since $\overline{\sigma}_V$ is an isomorphism by Lemma \ref{aug}, it follows that
$T^{1_\II}(V)$ is (isomorphic to) the equaliser  of the pair
$$\xymatrix{V \ot T(\II)
\ar@{->}@<0.5ex>[rr]^-{V \ot T(\II) \ot e_{\II}} \ar@{->}@<-0.5ex>
[rr]_-{V \ot \chi_{\II, \II}}&& V \ot T(\II)\otimes T(\II)}.$$

Using now that
\begin{itemize}
  \item $\theta \cdot e_\II=1_\II$ and $(\theta \ot T(\II))\cdot \chi_{\II,\II}=1_{T(\II)}$, since 
  the monad $\bT$ is opmonoidal, and
  \item $(\theta \ot T(\II))\cdot (T(\II) \ot e_{\II})=e_\II \cdot \theta$ by naturality of composition,
\end{itemize} one sees that the diagram
$$\xymatrix{\II \ar[r]^-{e_\II}&T(\II) \ar@/^1pc/[l]^\theta
\ar@{->}@<0.5ex>[rr]^-{T(\II) \ot e_{\II}} \ar@{->}@<-0.5ex>
[rr]_-{\chi_{\II, \II}}&& T(\II)\otimes T(\II)\ar@/^1.5pc/[ll]^{\theta \ot T(\II)}}.$$
is a split equaliser. It follows --since split equalisers are preserved by any functor-- that
the diagram 
$$\xymatrix{V \ar[r]^-{V \ot e_{\II}}& V \ot T(\II)
\ar@{->}@<0.5ex>[rr]^-{V \ot T(\II) \ot e_{\II}} \ar@{->}@<-0.5ex>
[rr]_-{V \ot \chi_{\II, \II}}&& V \ot T(\II)\otimes T(\II)}$$ is a (split) equaliser. Thus the
monad $\textbf{T}^{1_\II}$ is (isomorphic to) the identity monad.

Next, as $\sigma:T \to 1_\bV$ is a morphism of monads, one has in particular
that $\sigma \cdot e=1$. Thus the unit of the monad $\textbf{T}$ is a split monomorphism, and since 
the category $\bV$ is Cauchy complete by hypothesis, it follows from \cite[Corollary 3.17]{Me}
that $\textbf{T}$ is of effective descent type, i.e. the functor $\phi_T$ is comonadic. 
Putting now this information together, the assertions follow by
Theorem \ref{Main1}.
\end{proof}

\section{Applications}

In this final section we outline some applications of the notions developed. 

\begin{thm}\label{cart.}{\bf Monads on cartesian monoidal categories.} 
\em Let $\A$ be a category with finite products. Then
$\A$ is equipped with the (symmetric) monoidal structure $(\A,\times, \mathtt{1})$
(known as the \emph{cartesian monoidal structure}), where $a
\times b$ is some chosen product of $a$ and $b$, and $\mathtt{1}$ is
a chosen terminal object in $\A$. For any object $a \in \A$, we write $!_a$ for the
unique morphism $a \to \mathtt{1}$. Given morphisms $f: a \to x$ and $g: a \to y$ in $\A$, we write
$<f,g>:a \to x \times y$ for the unique morphism inducing commutativity of 
the diagram
$$\xymatrix{&& a \ar[lld]_{f} \ar[rrd]^{g} \ar[d]|{<f,g>}&&\\
x && x \times y \ar[ll]^{p_1} \ar[rr]_{p_2}&& y.}$$  

Any monad $\textbf{T}$ on $\A$ has a canonical structure of an opmonoidal monad  given by
$$\chi_{a,b}=<T(p_1),T(p_2)>:T(a \times b) \to T(a) \times T(b),$$
$$\theta=!_{T(\mathtt{1})}:T(\mathtt{1}) \to \mathtt{1}.$$ Thus, for any monad $\textbf{T}$ on $\A$,
the category $\A_T$ is also cartesian.

Since, for any $a \in \A$, the projection
$p_1:a\simeq a \times \mathtt{1} \to a$ is (isomorphic to) the
identity morphism $1_a: a \to a$, while the projection $p_2:a\simeq a \times \mathtt{1}\to \mathtt{1}$ is
(isomorphic to) the morphism $!_a : a \to \mathtt{1}$, $\chi_{a,\mathtt{1}}:T(a) \to T(a) \times T(\mathtt{1})$
is just the morphism $<1_{T(a)}, T(!_a)>$.

An arbitrary object $a\in \A$ has
a canonical $\A$-comonoid structure given by the diagonal
morphism $\Delta_a=<\!1_a, 1_a \!> : a \to a \times a$. Writing $\textbf{a}$ for the corresponding
$\A$-comonoid, one has that $\A^{\textbf{a}}$ is
(isomorphic to) the comma category  $\A\!\downarrow \!a$ (see, for example, \cite{MW}). Modulo this isomorphism,
the forgetful functor $U^{\textbf{a}}:\A^{\textbf{a}} \to \A$ corresponds to the functor
$$\Sigma_a:\A\!\downarrow \!a \to \A,\,\,\, (x \to a) \longrightarrow x,$$ 
while its right adjoint $\phi_{\textbf{a}}:\A \to \A^{\textbf{a}}$
corresponds to the functor
$$a^*:\A\to \A\!\downarrow \!a, \,\,\, x\longrightarrow (p_1: a \times x \to a).$$
\end{thm}

Suppose now  $\textbf{T}$ to be a monad on a cartesian category $\A$ such that the category $\A_T$ admits equalisers. 
Then one can form the monad $T^{1_\mathtt{1}}$. Moreover, modulo the isomorphism 
of categories 
$(\A_T)^{\phi_T(\mathtt{1})} \simeq(\A_T\downarrow \phi_T(\mathtt{1}))$, 
one rewrites Diagram \ref{D.1} from \ref{G.E} as 

\begin{equation}\label{D.3}
\xymatrix{\A \ar@/^2pc/@{->}[rrrr]^{K_{1_\mathtt{1},\mathtt{1}}}\ar[rrd]_{\phi_T} \ar[rr]^{\phi_{T^{1_\mathtt{1}}}}& &
\A_{T^{1_\mathtt{1}}}\ar[d]^{i_*} \ar[rr]^{\overline{i}} & &
(\A_T)\downarrow{{\phi_T(\mathtt{1})}}  \ar[lld]^{\Sigma_{{\phi_T(\mathtt{1})} }}\\
 & &  \A_T & . &}
 \end{equation} Note that, for any $a \in \A$, $K_{1_\mathtt{1},\mathtt{1}}(a)=((T(a),m_a),T(!_a))$.

\begin{thm}\label{R1}{\bf Remark.} \em Obviously, for any  $(a, h_a)
\in \A_T$, the $(a, h_a)$-component of the natural transformation
$\mathbb{H}^r_{-,\mathtt{1}}:T \to - \times T(\mathtt{1})$ is the
composite
$$\xymatrix{T(a) \ar[rr]^-{<1_{T(a)}, T(!_a)>} && T(a) \times T(\mathtt{1})
\ar[rr]^-{h_a \times T(\mathtt{1})}&&a \times T(\mathtt{1}),}$$ which is the same as the
morphism $$T(a) \xrightarrow{<h_a, T(!_a)>}a \times T(\mathtt{1}).$$ If $T(\mathtt{1}) \simeq \mathtt{1},$
then $T(!_a) \simeq !_{T(a)}$ and thus $<\!\!h_a, T(!_a)\!>$ can be identified with the morphism $h_a : T(a) \to a$.
\end{thm}

 Now fix a monad $\textbf{T}=(T,m,e)$ on a cartesian monoidal category $\A$ with equalisers.
Then, for any $a \in \A$, $T^{1_\mathtt{1}}(a)$ can be calculated as the equaliser of the diagram
$$\xymatrix{ T(a)
\ar@{->}@<0.5ex>[rr]^-{<1_{T(a)}, T(!_a)>} \ar@{->}@<-0.5ex>
[rr]_-{1_{T(a)} \times e_\mathtt{1}}&& T(a) \times T(\mathtt{1}).}$$ But since $1_{T(a)} \times e_\mathtt{1}$
can be identified with the morphism $<1_{T(a)},e_{\mathtt{1}}\cdot !_{T(a)}>$, the diagram
$$\xymatrix{T^{1_\mathtt{1}}(a)\ar[r]^{i_a}&T(a)
\ar@{->}@<0.5ex>[rr]^-{<1_{T(a)}, T(!_a)>} \ar@{->}@<-0.5ex>
[rr]_-{<1_{T(a)},e_{\mathtt{1}}\cdot !_{T(a)}>}&& T(a) \times T(\mathtt{1})}$$ is an equaliser if and only if so is
$$\xymatrix{T^{1_\mathtt{1}}(a)\ar[r]^{i_a}&T(a)
\ar@{->}@<0.5ex>[rrr]^-{p_2 \cdot <1_{T(a)}, T(!_a)>} \ar@{->}@<-0.5ex>
[rrr]_-{p_2 \cdot <1_{T(a)},e_{\mathtt{1}}\cdot !_{T(a)}>}&&& T(\mathtt{1})}.$$
 As $p_2 \cdot <1_{T(a)}, T(!_a)>=T(!_a)$
and $p_2 \cdot<1_{T(a)},e_{\mathtt{1}}\cdot !_{T(a)}>=e_{\mathtt{1}}\cdot !_{T(a)}$, the diagram
$$
\xymatrix{T^{1_\mathtt{1}}(a)\ar[r]^{i_a}&T(a)
\ar@{->}@<0.5ex>[rr]^-{ T(!_a)} \ar@{->}@<-0.5ex>
[rr]_-{e_{\mathtt{1}}\cdot !_{T(a)}}& & T(\mathtt{1})}$$ 
is an equaliser. It follows that if $T(\mathtt{1})\simeq \mathtt{1},$
then $i_a$ is an isomorphism. Conversely, if $i_a$ is an isomorphism, then $T(!_a)=e_{\mathtt{1}}\cdot !_{T(a)}$. In particular,
$T(!_\mathtt{1})=e_{\mathtt{1}}\cdot !_{T(\mathtt{1})}$. But $T(!_\mathtt{1})=1_\mathtt{1}$, implying that both $e_{\mathtt{1}}$ and $!_{T(\mathtt{1})}$
are isomorphisms. Thus:

\begin{lemma}\label{term.} Let $\textbf{T}$ be a monad on a cartesian monoidal category $(\A, \times, {\mathtt{1}})$. Then 
the canonical inclusion $i:T^{1_\mathtt{1}} \to T$  is an isomorphism if and only if the functor part $T$ preserves the terminal object.
\end{lemma}

\begin{proposition}\label{cart.mon.} Let $(\A, \times, {\mathtt{1}})$ be a cartesian monoidal category.
For any monad $\bT$ on  $\A$, whose functor part
preserves the terminal object ${\mathtt{1}}$, the comparison functor $$\overline{i}:\A_{T^{1_\mathtt{1}}}\to (\A_T)^{\phi_T(\mathtt{1})}$$
 is an equivalence.
In particular, $1_{\mathtt{1}} : {\mathtt{1}} \to {\mathtt{1}}$ is a Galois grouplike element \emph{(}w.r.t. $\bT$\emph{)}.
\end{proposition}
\begin{proof} Since $T(\mathtt{1})\simeq \mathtt{1}$, the monads $\textbf{T}^{1_\mathtt{1}}$ and $\bT$ are isomorphic
by Lemma \ref{term.}. It then follows from commutativity of Diagram \ref{D.3} that $\overline{i}$ is just the functor
$(\phi_T(\mathtt{1}))^*:\A_T \to (\A_T)\downarrow \phi_T(\mathtt{1}).$ But since $T(\mathtt{1})\simeq \mathtt{1}$, $\phi_T(\mathtt{1})$
is a terminal object in $\A_T$. Thus the functor $(\phi_T(\mathtt{1}))^*$ (and hence also $\overline{i}$) is an isomorphism of categories. 

Now the last assertion follows from Theorem \ref{Main}.
\end{proof}
 
Recall that a monad $\textbf{T}=(T,m,e)$ on a category $\A$ is said to be \emph{idempotent} if
the multiplication $m: TT \to T$ is a natural isomorphism.

\begin{proposition} \label{idem.} Let $(\A, \times, \mathtt{1})$ be a cartesian monoidal category. Any idempotent monad on $\A$, whose functor-part
preserves the terminal object ${\mathtt{1}}$, is right pre-Hopf.
\end{proposition}
\begin{proof}It is well-known that if $\textbf{T}=(T,m,e)$ is an idempotent monad on a category $\A$, then for any $(a, h_a)\in \A_T$,
the morphism $h_a : T(a) \to a$ is an isomorphism. Thus the result follows from Remark \ref{R1}.
\end{proof}

\begin{thm}\label{ex.2}{\bf Example.} \em Recall \cite{Mc} that a category $\A$ with all finite products is
called \emph{cartesian closed} if for each object $a\in \A$, the functor 
$$ a \times - :
\A \to \A$$ has a right adjoint $$ (-)^a : \A \to \A.$$

It is well known that %, for any object $a \in \A$,
the endofunctor can be made a monad $T_a=(-)^a$ with multiplication and unit
$$m_x=x^{\Delta_a} :T_aT_a(x)=(x^a)^a\simeq x^{a \times a} \to T_a(x)=x^a,$$
$$e_x=x^{!_a} : x \to x^a=T_a(x).$$

Let $\A$ be a cartesian closed category such that the terminal object $\mathtt{1}$ has a nontrivial
proper subobject $u\rightarrowtail \mathtt{1}$ (for example, let $\A$ be the category of set-valued
sheaves on a nontrivial topological space). Since $u \times u \simeq u$, the diagonal $\Delta_u: u \to u\times u$
is an isomorphism, whence the monad $T_u$ is idempotent. Since $\mathtt{1}^u=\mathtt{1}$, the functor $(-)^u$ preserves
the terminal object and it follows from Proposition \ref{idem.} that the opmonoidal monad $T_U$ is right pre-Hopf. 

Note that by Proposition \ref{cart.mon.}, the comparison functor
$K_{1_{\mathtt{1}},\mathtt{1}}: \bV \to (\V_{\emph{T}_u})^{^{\phi_{\emph{T}_u}(\mathtt{1})}}$ is not an
equivalence of categories. Thus $T_u$ is an example of an opmonoidal monad which is right pre-Hopf, but the
corresponding comparison functor $K_{1_{\mathtt{1}},\mathtt{1}}$ is not an equivalence of categories.
\end{thm}

\begin{thm}\label{ex.3}{\bf Example.} \em Recall that the \emph{covariant power set functor} $\mathcal{P} : \text{Set} \to \text{Set}$ is defined by
$$\mathcal{P}(X)=\text{Sub}(X),\quad \mathcal{P}(f:X \to Y)=\mathcal{P}(X) \xrightarrow{\mathcal{P}(f)}\mathcal{P}(Y), $$
where $\text{Sub}(X)$ is the set of all subsets of $X$ and for each $U \in \text{Sub}(X)$, $\mathcal{P}(f)(U)$ is the 
image $f(U)$ of $U$ under $f$. $\mathcal{P}$ is actually the functor part of a monad $(\mathcal{P},e,m)$ 
with 
\begin{center}
$e_X: X \to \mathcal{P}(X)$ the singleton map, $e_X: x \to \{x\}$, and  \\
$m_X: \mathcal{P}\mathcal{P}(X)\to \mathcal{P}(X)$
 the union, $m_X(\{X_\alpha\})=\bigcup_\alpha X_\alpha.$ 
\end{center}

It is well-known that the Eilenberg-Moore category of
$\mathcal{P}$-algebras is isomorphic to the category $\textbf{CSLat}$ of \emph{complete (join-)semilattices}. Recall that the category
$\textbf{CSLat}$ has as its objects partially ordered sets $(X, \leq)$ which admit arbitrary suprema, and as its morphisms 
$f:X \to Y$ maps which preserve suprema. We write $\mathtt{2}$ for the two-element semilattice $\phi_{\mathcal{P}}(\mathtt{1})=\{0\leq 1\}$. 

It is not hard to check that $\mathcal{P}^{1_\mathtt{1}}$ is just the \emph{proper power set functor} $\mathcal{P}^+$, where 
$\mathcal{P}^+(X)=\mathcal{P}(X)\setminus \{\varnothing\}$. It is also well-known (see, for example, \cite[Problem 1.3.3.]{H}) that
the Eilenberg-Moore category of
$\mathcal{P}^+$-algebras is isomorphic to the category $\textbf{ACSLat}$ of \emph{almost complete (join-)semilattices}, 
i.e. partially ordered sets $(X, \leq)$ such that the suprema of all non-empty subsets of X exists. 
Morphisms $f : (X, \leq) \to (Y, \leq)$ of $\textbf{ACSLat}$ are non-empty suprema preserving maps.

Writing $i:\mathcal{P}^+ \to \mathcal{P}$ for the canonical inclusion, it is not hard to see that the functor
$$i^* :\text{Set}_\mathcal{P}=\textbf{CSLat} \to  \text{Set}_{\mathcal{P}^+}=\textbf{ACSLat}$$
 just forgets about
the bottom element, while 
$$i_*:\textbf{ACSLat} \to \textbf{CSLat}$$
 takes an object $X \in \textbf{ACSLat}$ to
the complete semilattice  $\overline{X}$ obtained from $X$ by adding a bottom element $0_X$. It then follows in particular
that the endofunctor $i_*i^*:\textbf{CSLat} \to \textbf{CSLat}$ takes a complete semilattice $X$ to the complete
semilattice $\overline{X}$ obtained from $X$ by adding a new bottom element $0_{\overline{X}} < 0_X.$ Direct inspection shows that,
for any $X \in \textbf{CSLat}$, the $X$-component of the comonad morphism 
$S_{\overline{i}}:\textbf{G}_i \to \textbf{G}_{\phi_{_\mathcal{P}}(\mathtt{1})}$ is the map $\omega:\overline{X}\to X\times \mathtt{2}$
defined by 
$$\omega(x)=\left\{
\begin{array}{ll}
(x,1)& \text{if} \,\, x\neq o_{\overline{X}}\\
(0_X, 0) & \text{if} \,\, x= o_{\overline{X}} .
\end{array}\right.$$ 
It is clear that $\omega$ is not an isomorphism. Thus $1_{\mathtt{1}} : {\mathtt{1}} \to {\mathtt{1}}$ is 
not a Galois grouplike element (w.r.t. the monad $\mathcal{P}$), and hence by Theorem \ref{Main} the comparison functor
$$\overline{i}:\text{Set}_{\mathcal{P}^+}=\textbf{ACSLat} \to (\text{Set}_\mathcal{P}\downarrow \phi_{_\mathcal{P}}(\mathtt{1}))
=(\textbf{CSLat}\downarrow \mathtt{2}),$$ which sends an object $X \in \textbf{ACSLat}$ to 
$(\omega: \overline{X}\to \mathtt{2}) \in (\textbf{CSLat}\downarrow \mathtt{2})$ with 
$$\omega(x)=\left\{
\begin{array}{ll}
1& \text{if} \,\, x\neq o_{\overline{X}}\\
0 & \text{if} \,\, x= o_{\overline{X}} ,
\end{array}\right.
$$ is not an equivalence of categories. According to \cite[1.4]{MW}, $\overline{i}$ admits a right adjoint $r$:
for any $(\omega: \overline{X}\to \mathtt{2}) \in (\textbf{CSLat}\downarrow \mathtt{2})$, $r(\omega)=(\omega)^{-1}(1)$.
It is now easy to see that $r\overline{i} \simeq 1$. Thus $\textbf{ACSLat}$ is a full coreflective subcategory of
$(\textbf{CSLat}\downarrow \mathtt{2})$.

Note finally that $\mathcal{P}^+(\mathtt{1})=\mathtt{1}$. Now it follows from Proposition \ref{cart.mon.} that  
$1_{\mathtt{1}} : {\mathtt{1}} \to {\mathtt{1}}$ is a Galois grouplike element w.r.t. the monad $\mathcal{P}^+$.
\end{thm}

\bigskip

\noindent
{\bf Addresses:} \\[+1mm]
{Razmadze Mathematical Institute, 1, M. Aleksidze st., Tbilisi
0193,  } {\small and} \\
 {Tbilisi Centre for Mathematical Sciences,
Chavchavadze Ave. 75, 3/35, Tbilisi 0168}, \\
 Republic of Georgia,
    {\small bachi@rmi.acnet.ge}\\[+1mm]
{Department of Mathematics of HHU, 40225 D\"usseldorf, Germany}, \\
  {\small wisbauer@math.uni-duesseldorf.de}


\begin{thebibliography}{55}

\bibitem{BBW} B\"ohm, G., Brzezi\'nski, T. and Wisbauer, R.,
    {\em Monads and comonads on module categories},
  J. Algebra 322, 1719-1747 (2009).

 \bibitem{BV} Brugui\`eres, A. and Virelizier, A., {\em Hopf monads},
Adv. Math. 215(2), 679-733 (2007).

\bibitem{BLV}   Brugui\`{e}res, A., Lack, S. and Virelizier, A.,
    {\em Hopf monads on monoidal categories}, preprint,
   arXiv:1003.1920v3 (2010).

\bibitem{D} Dubuc, E., {\em Kan extensions in enriched category theory},
LN Math. 145, Berlin-Heidelberg-New York,
Springer-Verlag (1970).

\bibitem{GoTo} G\'omez-Torrecillas, J.,
 {\em Comonads and Galois corings},
 Appl. Categ. Struct. 14(5-6), 579-598 (2006).

\bibitem{H} H\"{o}hle, U., {\em Many valued topology and its applications},
 Kluwer Academic Publishers (2001).

\bibitem{McC} McCrudden, P., {\em Opmonoidal monads},
     Theory Appl. Categ. 10, 469-485 (2002).

\bibitem{M}    Mesablishvili, B.,  {\em Entwining Structures in Monoidal
Categories},  J. Algebra  319, 2496-2517 (2008).

\bibitem{Me} Mesablishvili, B., {\em  Monads of effective descent type
   and comonadicity},  Theory Appl. Categ. 16, 1-45 (2006).

\bibitem{MW1}  Mesablishvili, B. and Wisbauer, R.,
   {\em Bimonads and Hopf monads on categories},  
   J. K-Theory 7(2), 349-388 (2011).

\bibitem{MW}  Mesablishvili, B. and Wisbauer, R.,
 {\em Galois functors and entwining structures},
 J. Algebra  324, 464-506 (2010).

 \bibitem{Mc} S. MacLane, {\em Categories for the Working Mathematician},
Graduate Texts in Mathematics Vol. 5, Springer, Berlin-New York,
1971.

 \bibitem{Moer} Moerdijk, I.,
   {\em Monads on tensor categories},
   J. Pure Appl. Algebra 168(2-3), 189-208 (2002).


\bibitem{W}  Wisbauer, R.,  {\em Algebras versus coalgebras},
    Appl. Categ. Struct. 16(1-2), 255-295 (2008).

\bibitem{Wf}   Wolff, H., {\em V-Localizations and V-monads}.
 J. Algebra  24, 405-438 (1973).

\end{thebibliography}
\end{document}